\documentclass[11pt]{amsart}
\usepackage{epsfig,color}
\usepackage{xy}
\xyoption{matrix}
\xyoption{arrow}

\newtheorem{lemma}{Lemma}[section]
\newtheorem{theorem}[lemma]{Theorem}
\newtheorem{corollary}[lemma]{Corollary}

\newtheorem{proposition}[lemma]{Proposition}

\newcommand{\F}{\mathcal F}

\newcommand{\Hom}{\operatorname{Hom}}

\newcommand{\Ext}{\operatorname{Ext}}
\newcommand{\thick}{\operatorname{thick}}
\newcommand{\End}{\operatorname{End}}

\renewcommand{\ker}{\operatorname{ker}}
\newcommand{\coker}{\operatorname{coker}}
\newcommand{\add}{\operatorname{add}}
\newcommand{\module}{\operatorname{mod}}

\newcommand{\U}{\mathcal U}
\newcommand{\C}{\mathcal C}
\newcommand{\D}{\mathcal D}

\newcommand{\E}{\mathcal E}
\newcommand{\Q}{\mathcal Q}
\renewcommand{\S}{\mathcal S}

\renewcommand{\mod}{\operatorname{mod}}

\title{Three kinds of mutation}
\author[Buan]{Aslak Bakke Buan}
\address{Institutt for matematiske fag\\
Norges teknisk-naturvitenskapelige universitet\\
N-7491 Trondheim\\
Norway}
\email{aslakb@math.ntnu.no}

\author[Reiten]{Idun Reiten}
 \address{Institutt for matematiske fag\\
Norges teknisk-naturvitenskapelige universitet\\
N-7491 Trondheim\\
Norway}
\email{idunr@math.ntnu.no}

\author[Thomas]{Hugh Thomas}
\address{Department of Mathematics and Statistics\\University of 
New Brunswick\\Fredericton NB\\E3B 1J4 Canada
}
\email{hthomas@unb.ca}

\date{\today}
\begin{document}

\begin{abstract}For a finite dimensional hereditary algebra $H$, we consider: exceptional
sequences in $\mod H$, silting objects in the bounded derived category $\D$, and $m$-cluster tilting objects
in the $m$-cluster category, $\C_m=\D/\tau^{-1}[m]$.
There are mutation operations on both the set of $m$-cluster tilting objects and 
the set of exceptional sequences.
It is also possible
to define a mutation operation for silting objects.  
We compare these three different notions of mutation. 
\end{abstract}

\maketitle

\section*{Introduction}

Let $H$ be a hereditary finite dimensional algebra with $n$ isomorphism
classes of simple objects.  The tilting objects
for $H$ inside $\D=D^b(H)$, the bounded derived category of finitely generated
$H$ modules, are of considerable importance, and several
variations on the concept have also been studied.  In
this paper, we will be concerned with three of these: exceptional
sequences in $\mod H$, silting objects in $\D$, and $m$-cluster tilting objects
in the $m$-cluster category, $\C_m=\D/\tau^{-1}[m]$, which is an orbit 
category of $\D$.

There are mutation operations on both the set of $m$-cluster tilting objects and 
the set of exceptional sequences.
It is also possible
to define a mutation operation for silting objects.  
Our main goal in this paper
is to relate these different notions of mutation.  
In order to state our results, we must introduce some definitions.  

A basic object $T$ in $\D$ is said to be partial silting if
$\Ext^{i}(T,T)= 0$ for $i>0$, and silting if in addition $T$ is maximal with this property.  
Partial silting objects were also studied in \cite{kv} 
(they were then called silting objects), and in \cite{ast}, where some 
connections to exceptional sequences were also investigated.

A basic object $T$ in $\C_m$
is called $m$-cluster tilting if $\Ext_{\C_m}^{i}(T,T)= 0$ for $0 < i \leq m$ and 
if $X$ satisfies 
$\Ext_{\C_m}^{i}(T,X)= 0$ for $0 <i \leq m$, then $X$ is in $\add T$.
This condition is known to be equivalent to
$T$ being maximal with respect to $\Ext_{\C_m}^{i}(T,T)= 0$ for $0 < i \leq m$ 
\cite{w,zz}, and to $T$ having $n$ indecomposable summands.  If $T$ satisfies 
$\Ext_{\C_m}^{i}(T,T)=0$ for $0<i\leq m$ and it has $n-1$ summands, it is called
{\em almost $m$-cluster tilting}.  

If we make a (reasonable, but not canonical) choice of a fundamental domain
in $\D$ for $\tau^{-1}[m]$, then the lifting of objects from 
$\C_m$ to $\D$ takes $m$-cluster tilting objects to silting objects.  
Conversely, any silting object is the lift of an $m$-cluster tilting object
for $m$ sufficiently large.  This allows us to deduce the number of 
summands of a silting object, and to lift the notion of mutation from
$m$-cluster tilting objects to silting objects.  

Exceptional sequences were 
investigated first 
in the context of algebraic geometry (see \cite{rudakov}), and
later for hereditary finite dimensional algebras in \cite{cb,r}.
An exceptional sequence in $\mod H$ is a sequence of indecomposable
$H$ modules, $(E_1,\dots,E_r)$, with the properties that $\Hom(E_i,E_j)=0$ 
for $j<i$ and $\Ext^1(E_i,E_j)=0$ for $j\leq i$.  The maximal length of
an exceptional sequence is $n$; an exceptional 
sequence of maximal length is called 
{\em complete}, and one of length $n-1$ is called {\em almost complete}.

Given a sequence of indecomposable $H$ modules $(M_1,\dots,M_n)$, a {\em
placement} of $(M_1,\dots,M_n)$ is a sequence $(\widehat M_1,\dots,\widehat
M_n)$ of objects in $\D$ such that $\widehat M_i$ is isomorphic to 
$M_i[t_i]$ for some $t_i\in \mathbb Z$.  
We will be interested in studying placements of an exceptional 
sequence which
form a silting object. 
It is easy to see that any 
complete 
exceptional sequence admits such a placement.  
It is natural to ask if it is possible to define a placement
for each exceptional sequence in $\mod H$, such that if two
exceptional sequences are related by a single mutation, the
same would be true for  the corresponding silting objects.  
Simple examples show that this is too much to ask.  In fact,
for some $H$, 
there exists an exceptional sequence $\E$ such 
that it is not possible to place $\E$ and the sequences $\E_i$ obtained by
applying a single mutation to $\E$, in such a way that the 
corresponding silting object $\widehat \E$ is related
by a single mutation to each of the $\widehat \E_i$.  

Our main result is as follows.   
Let $\F$ be an almost complete 
exceptional sequence, and $\E_1,\dots,\E_{n+1}$ be the complete
exceptional sequences which can be obtained by adding one term to
$\F$.  Essentially by definition, $\E_i$ and $\E_{i+1}$ are related by a 
single mutation.  Then there exists a placement $\widehat \E_i$ of each $\E_i$, all of which
agree as to the placement of $\F$, such that $\widehat \E_i$ is silting
for all $i$, and such that $\widehat \E_i$ and $\widehat \E_{i+1}$ are 
related by a single mutation.  

In the course of our investigations, we define a quiver associated to an
exceptional sequence, as follows.  The vertex set consists of the 
terms in the exceptional sequence; there is an arrow from $E_i$ to 
$E_j$ if $\Hom(E_i,E_j)\ne 0$ or $\Ext^1(E_j,E_i)\ne 0$.  (Note the reversal
of the order of $E_i$ and $E_j$ in the two conditions.)  We show that 
this quiver is always acyclic, and that it is connected provided $H$ is connected and the exceptional sequence is complete.  This acyclicity is crucial for our main result mentioned above.

In section 1,
we discuss exceptional sequences, starting with definitions and basic properties and we define and study
the quiver associated with an exceptional sequence.
In section 2 we give connections between exceptional sequences, silting objects and $m$-cluster tilting objects and
in section 3 we consider connections between the different kinds of
mutation. 
Finally, in section 4 we show that for a fixed almost complete
  exceptional sequence $\E$, and its set of complements, there is a
  natural interpretation of $\E$ as an almost complete silting object in
  such a way that complements and exchange sequences correspond.

\section{Properties of exceptional sequences}\label{sec1}

In this section we first recall the definition and some basic properties of exceptional
sequences. Then we define, for each exceptional sequence $\E$, 
a naturally associated quiver $\Q_{\E}$, and show that such quivers are always 
acyclic, and are connected provided $H$ is connected and $\E$ is complete. 
In section \ref{sec4}, these quivers and their acyclicity will be a main ingredient
in the proof of Theorem \ref{main}.

As before, let $H$ be a finite dimensional hereditary algebra over a field $k$.
We do not assume that $k$ is algebraically closed.  
We always assume that $H$ has $n$ simple modules up to isomorphism.
An indecomposable module $E$ is called {\em exceptional} if $\Ext^1(E,E) = 0$, and
a sequence $(E_1, \dots , E_r)$ of exceptional $H$-modules is said to be an {\em exceptional sequence}
if $\Hom(E_j,E_i) = 0 = \Ext^1(E_j,E_i)$ for $j > i$.

\subsection{Basic properties}\label{basics}

We have the following useful fact.

\begin{lemma}\label{atmostone}
Given an exceptional sequence $(E_1, \dots , E_r)$,
then $\Hom(E_i,E_j)$ and $\Ext^1(E_i,E_j)$ cannot both be non-zero for $i<j$.
\end{lemma}

\begin{proof}
Fix $E=E_i$, $F=E_j$.  
By results from \cite{cb,r}, we can consider $({E}, {F})$ as an exceptional sequence
in a hereditary module category, say $\mod H'$, with $H'$ of rank 2 and such that $\mod H'$ has a full 
and exact embedding into $\mod H$.
For a hereditary algebra $H'$ of rank 2, the only exceptional indecomposable modules are
preprojective or preinjective. Hence, a case analysis of the possible exceptional sequences in
$\mod H'$ for such algebras, gives the result.
\end{proof}

There are right and left mutation operations on exceptional sequences. Here we use {\em right mutation} $\mu_i$,
which for an exceptional sequence $\E = (E_1,E_2, \dots ,E_r)$ and an index $i \in \{1, \dots, r-1 \}$ is 
defined by replacing the pair $(E_i,E_{i+1})$ with a pair
$(E_{i+1}, E_i^{\ast})$. Here $E_i^{\ast}$ is defined as follows. 
If $\Hom(E_i,E_{i+1}) \neq 0$, consider the minimal left $\add E_{i+1}$-approximation $f \colon E_i \to E_{i+1}^t$.
Then $f$ is either an epimorphism or a monomorphism (by \cite{hr,rs}). If $f$ is an epimorphism, let 
$E_i^{\ast} = \ker f$. If $f$ is a monomorphism, let $E_i^{\ast} = \coker f$.
If $\Ext^1(E_i,E_{i+1}) \neq 0$, then $ E_i^{\ast}$ is defined by the universal extension \cite{bong}
$$0 \to E_{i+1}^s \to E_i^{\ast} \to E_i \to 0.$$ 
Finally, if $\Hom(E_i,E_{i+1}) = 0 = \Ext^1(E_i,E_{i+1})$,
then $E_i^{\ast} = E_i$. 

One obtains a uniform description of $E_i^{\ast}$ as follows. Let $\thick E_{i+1}$ be the
thick subcategory of $\D$ generated by $E_{i+1}$. The indecomposable objects in this subcategory are of the form $E_{i+1}[j]$, for some
integer $j$. Let $E_i \to Z$ be a minimal left $(\thick E_{i+1})$-approximation in $\D$, and complete to a triangle
$$E_i \to Z \to E^{\sharp}.$$ 
It is easy to see that $E_i^{\sharp} = E_i^{\ast}[j]$ with $j$ either $0$ or $1$.

We recall the following from \cite{cb,r}.

\begin{proposition}\label{known}
Let $\E = (E_1,E_2, \dots ,E_r)$ be an exceptional sequence
in $\module H$. 
\begin{itemize}
\item[(a)] We have $r \leq n$.
\item[(b)] If $r<n$, then there is an complete exceptional sequence 
$$(E_1,E_2, \dots ,E_r, E_{r+1}, \dots E_n).$$
\item[(c)] If $r=n-1$, then there is for each $j \in \{1, \dots, n\}$
a unique indecomposable module $M$,  such that 
$$(E_1, \dots ,E_{j-1}, M, E_j, \dots, E_{n-1})$$ is an exceptional sequence.
\item[(d)] There is an ordering $S_1, \dots, S_n$ of the simple $H$-modules,
such that $\S = (S_1, \dots, S_n)$ is a complete exceptional sequence. Moreover,
any complete exceptional sequence can be reached from $\S$ by a sequence of mutations and inverses of mutations.
\end{itemize}
\end{proposition}

\subsection{The Hom-Ext-quiver}

We now associate a quiver $\Q_{\E}$ to an exceptional sequence $\E= (E_1,E_2, \dots ,E_r)$.
The vertices of $\Q_{\E}$ are in bijection with $\{ E_1,E_2, \dots ,E_r \}$.
There is an arrow $E_i \to E_j$ if $\Hom(E_i, E_j) \neq 0$ or if $\Ext^1(E_j, E_i) \neq 0$.
We call this the {\em Hom-Ext quiver} of $\E$.

The construction of this quiver was motivated by the situation when 
$H$ is representation finite.  In that case, the AR quiver defines a 
partial order on the indecomposable $H$-modules, which can 
be restricted to give a partial order on the terms of $\E$.  
If $E_i\rightarrow E_j$
in the Hom-Ext quiver, then $E_i$ precedes $E_j$ with respect to this
partial order (though not conversely).  Where we have used the 
acylicity of $\Q_\E$ below, we could instead have used the order induced
from the AR quiver.      
When $H$ is not representation finite, however, the AR quiver
will have more that one components and will typically contain cycles, making it
difficult to use it to define a partial order on the terms of $\E$; 
the Hom-Ext quiver is a replacement for this order.  

Note that if $A$ and $B$ are distinct exceptional modules with $\Hom(A,B) \neq 0$, then
there is, by \cite{hr}, either a monomorphism $A \to B$ or an epimorphism $A \to B$ and (by reason of total dimension) not both.
Let us add some decoration to our quiver.
There are three types of arrows. There is an m-arrow  $E_i \to E_j$ if 
there is a monomorphism $E_i \to E_j$.
There is an e-arrow  $E_i \to E_j$ if  there is an epimorphism $E_i \to E_j$, and 
there is an x-arrow  $E_i \to E_j$ if $\Ext^1(E_j, E_i) \neq 0$.

We aim to prove that the quiver $\Q_{\E}$ is acyclic and that it is connected
provided $\E$ is 
complete and $H$ is connected.
For this purpose the following result
is useful.

\begin{lemma}\label{elementary}
Assume $A,B,C$ are exceptional modules contained in an exceptional sequence $\E$.
Then we have the following.
\begin{itemize}
\item[(a)] If there is an e-arrow $A \to B$, then there is not an m-arrow  $B \to C$.
\item[(b)] If there are m-arrows $A \to B$ and $B \to C$, there must be an m-arrow
$A \to C$.
\item[(c)] If there are e-arrows $A \to B$ and $B \to C$, there must be an e-arrow
$A \to C$.
\item[(d)]  If there is an e-arrow $A \to B$ and an x-arrow $B \to C$,
then there is an  x-arrow $A \to C$.
\item[(e)] If there is an x-arrow $A \to B$ and an m-arrow $B \to C$,
then there is an x-arrow $A \to C$.
\item[(f)] If there are arrows $A \to B \to C$ and $B \to C$ is an m-arrow,
then there is an arrow $A \to C$. 
\end{itemize}
 \end{lemma}

\begin{proof}
(a) follows from the fact that the composition of an epimorphism and a monomorphism is a nonzero map which
is neither a monomorphism nor an epimorphism.

(b) and (c) are obvious.

For (d) apply $\Hom(C,\ )$ to the exact sequence 
$$0\to K \to A \to B \to 0$$
to obtain an epimorphism $\Ext^1(C,A) \to \Ext^1(C,B)$.

(e) is dual to (d).

For (f), we have that $A \to B$ cannot be an e-arrow by (a). Then the claim follows from
combining (b) and (e).
\end{proof}

Applying this we first show acyclicity.

\begin{theorem}\label{acyclic}
The Hom-Ext quiver $\Q_{\E}$ of a complete exceptional sequence is acyclic.
\end{theorem}

\begin{proof}
Assume there is a cycle in the quiver $\Q_{\E}$. We can assume that the cycle has minimal
length. Hence, for any two vertices $X$ and $Y$ on this cycle which are not consecutive
on the cycle, there are no arrows between $X$ and $Y$. Note also that by Lemma \ref{atmostone} the length
must be at least three. By Lemma \ref{elementary} (f) there cannot be any m-arrows on the minimal cycle.

It is clear that there must be at least one x-arrow on the cycle, and at least one arrow
which is not an x-arrow, since otherwise the objects corresponding to 
vertices on the cycles could not be ordered into an exceptional sequence.
Thus there must be an x-arrow which is preceded by an e-arrow. However,
by Lemma \ref{elementary} (e), this is not possible on a minimal cycle.
\end{proof}

We also have the following property. 

\begin{theorem} If $H$ is connected, 
then the Hom-Ext quiver $\Q_{\E}$ of a complete exceptional sequence $\E$ in $\module H$ is connected.
\end{theorem}

\begin{proof}
Assume that $Q_{\E}$ is not connected, and assume that $\E = \E_1 \cup \E_2$, with   
$\E_1 = \{A_1, \dots, A_r \}$ and $\E_2  = \{B_1, \dots, B_s \}$ being non-empty disjoint subsets of the elements in $\E$, 
with the property that there are no arrows between any vertex corresponding to an object in $\E_1$ 
and any vertex corresponding to an object in $\E_2$. 

This is equivalent to having $\Hom(A_i,B_j) = \Hom(B_j, A_i) = \Ext^1(A_i,B_j) = \Ext^1(B_j,A_i)= 0$
for all $i,j$.
We claim that any exceptional sequence $\E'$ obtained by mutating $\E$,
also has the property that $Q_{\E'}$ is disconnected. This will give a contradiction, because of transitivity
of the action of mutations on exceptional sequences, and Proposition \ref{known} (d).
To prove the claim consider the mutation pair $(E_i,E_{i+1})$, and consider the four possibilities:

\begin{itemize}
\item[(1)] $\Hom(E_i, E_{i+1}) = \Ext^1(E_i, E_{i+1}) =0$ and $E_i^{\ast} = E_i$
\item[(2)] There is an exact sequence $0 \to E_i \to E_{i+1}^t \to E_i^{\ast} \to 0$
\item[(3)] There is an exact sequence $0 \to E_i^{\ast} \to E_i \to E_{i+1}^t \to 0$
\item[(4)] There is an exact sequence $0 \to E_{i+1}^t \to E_i^{\ast} \to E_i \to 0$
\end{itemize}

Recall that in cases (2) and (3) the map $E_i \to E_{i+1}^t$ is a minimal left $\add E_{i+1}$-approximation, and
in case (4) the exact sequence is a universal extension.

We claim that if we replace
$E_i$ with $E_i^{\ast}$, then $Q_{\E'}$ will disconnect in the same way as $Q_{\E}$.
If we are in situation (1) above, this is trivial. Assume we are in situation (2), (3) or (4).
Then either both $E_i$ and $E_{i+1}$ are in $\E_1$ or both are in $\E_2$. Without loss of generality we 
assume they are both in $\E_1$. Let $B$ an object in $\E_2$. Considering the
long exact sequences obtained from applying $\Hom(B,\ )$ and $\Hom(\ ,B)$ to the relevant exact sequence (2),
(3) or (4), we obtain that
$\Hom( E_i^{\ast},B) = \Hom(B,  E_i^{\ast}) = \Ext^1( E_i^{\ast},B) = \Ext^1(B, E_i^{\ast})= 0$.
This gives the desired result.

\end{proof}

\section{Connection between exceptional sequences, silting objects and $m$-cluster tilting objects}

In this section we point out how to construct silting objects from exceptional sequences and vice versa. 
We also investigate the connection between 
silting objects and $m$-cluster tilting objects.

\subsection{Silting objects and exceptional sequences}

The results in this section partially overlap with \cite{ast}.
As before, let $H$ be a finite dimensional hereditary algebra with $n$ isomorphism classes of simples,
and let $T$ be an object in $\D$.
Recall that $T$ is called a {\em partial silting object} if $\Ext^i_{\D}(T,T) = 0$ for all $i>0$. 
Without loss of generality we can assume that $T$ is in the non-negative part $\D^+$of $\D$, i.e. we assume
\begin{equation}\label{eq} T= T_0[0] \oplus T_1[1] \oplus \cdots \oplus T_m[m] \end{equation}
with each $T_i$ in $\module H$. Then we have the following.

\begin{lemma}\label{silting}
$T$ is partial silting if and only if 
\begin{itemize}
\item[(i)]$\Ext^1_H(T_i,T_i) = 0$ for each i
\item[(ii)]$\Hom(T_i,T_j) = 0 = \Ext^1_H(T_i,T_j)$ for $i>j$.
\end{itemize}
\end{lemma}

\begin{proof}
We have that $T$ is partial silting if and only if $\Ext^1(T_i,T_i) = 0$ for each $i$ and 
$\Hom(T_i[i],T_j[j][t]) = \Hom(T_i[i],T_j[t+j])=0$ for $t>0$. 
Note that if $i <j$, then $\Hom(T_i[i],T_j[t+j]) = 0$ for $t>0$, since $t+j-i \geq 2$.
We also have that if $i >j$, then $\Hom(T_i[i],T_j[j][t]) \simeq \Hom(T_i,T_j[t-(i-j)])$.
Hence, for $i>j$, we have that 
$\Hom(T_i,T_j[t-(i-j)]) = 0$ for all $t>0$ if and only if $\Hom(T_i,T_j)= 0= \Ext^1(T_i,T_j)$.
This finishes the proof. 
\end{proof}

We now show the following connection between partial silting objects and  
exceptional sequences in $\module H$.

\begin{lemma}\label{silting2}
\begin{itemize}
\item[(a)]
Let $T_0,\dots,T_m$ be partial tilting objects in $\mod H$.  
As in (\ref{eq}), let $T=\bigoplus_{i=0}^m T_i[i]$, 
and let $\widetilde T=\bigoplus_{i=0}^m T_i$.
We have:
\begin{itemize}
\item[(i)] $T$ is a partial silting object if and only if the summands
of $\widetilde{T}$ can be ordered to be an exceptional sequence with all summands of $T_i$ preceding all summands
of $T_j$ for $i<j$.  
\item[(ii)]
$T$ is a silting object if and only if the summands of $\widetilde{T}$ can be ordered as in (i)
above to a complete exceptional sequence.
\end{itemize}
\item[(b)] Any silting object has exactly $n$ indecomposable direct
  summands, up to isomorphism.
\item[(c)]
Let $T' = T'_0 \oplus  T'_1 \oplus \cdots \oplus  T'_r$ be in $\module H$, with the $T_i'$ indecomposable and
mutually non-isomorphic, and let
$$\overline{T'} = T'_0[0] \oplus  T'_1[1] \oplus \cdots \oplus  T'_r[r]$$ be in $\D$. Then we have that  
\begin{itemize}
\item[(i)] $(T'_0, T'_1, \dots, T'_r)$ is an exceptional sequence if and only if
$\overline{T'}$ is a partial silting object.
\item[(ii)] $(T'_0, T'_1, \dots, T'_r)$ is a complete exceptional sequence if and only if
$\overline{T'}$ is a silting object. 
\end{itemize}
\end{itemize}
\end{lemma}

\begin{proof}
(a)(i) Assume $T$ is partial tilting. 
Then each $T_i$ is a partial tilting module 
over $H$, and hence the endomorphism algebras $\End_H(T_i)$  
has no oriented cycles \cite{hr}. It now follows from Lemma \ref{silting}
that the indecomposable direct summands of $T_i$ can be ordered as $T_i^{(1)}, \dots, T_i^{(s)}$,
with $\Hom(T_i^{(u)},T_i^{(v)}) =0$ for $u>v$.

\sloppy For the opposite implication, we assume there is an ordering 
such that $\Hom(T_i^{(u)},T_i^{(v)}) =0$ for $u>v$. We then check the conditions of Lemma~\ref{silting}. Condition 
(i) holds by assumption.  Condition (ii) follows from the definition of 
exceptional sequence.  

(ii) Assume that $T$ is a silting object, and that $\widetilde{T}$ can be ordered to an 
exceptional sequence $\E$. Assume $\E$ is not complete. Then
we can extend $\E$ to a new exceptional sequence by adding an object $U$ at the right end.
Then $T \oplus U[r+1]$ is a partial silting object by part (i), 
and we have a contradiction, so $\E$ must be complete.  

On the other hand, assume that $\widetilde{T}$ can be ordered to an exceptional sequence $\E_T$,
and that $T$ is partial silting but not silting. Then there is some $R[j]$ in $\D$, with $R$ not in $\add T$,
such that $T \oplus R[j]$ is partial silting. Note that $j$ might be one of the shifts already occurring in
the decomposition of $T$, in which case we add $R$ to $T_j$. Then, using (i), we get a contradiction to
the completeness of $\E_T$.

(b) This is a direct consequence of (a)(ii).

(c) Part (i) follows directly from Lemma \ref{silting}, while part (ii) follows
from (b) in combination with Lemma \ref{silting}.

\end{proof}

\subsection{Silting objects and maximal $m$-rigid objects}ß
Here we show that silting objects are closely related to $m$-cluster tilting objects 
in $m$-cluster categories.

Recall that for an integer $m \geq 1$, the $m$-cluster category of a finite dimensional hereditary
algebra $H$ is the orbit category $\C_m =\D / \tau^{-1}[m]$, where $\tau$ is the AR-translation in $\D$, see \cite{happelbook},
and $[m]$ is the $m$-fold composition of $[1]$, the shift functor. 
This category is known to be triangulated by \cite{k}. It was first studied in
the case $m= 1$ in \cite{bmrrt}. An object $T$ in $\C_m$ is said to be {\em $m$-rigid} if
$\Ext_{\C_m}^i(T,T) = 0$ for $i= 1, \dots, m$; and if $T$ is maximal with this property,
then it is said to be {\em maximal $m$-rigid}. 
This has been shown \cite{w,zz} to be equivalent to $T$ being $m$-cluster tilting, that is $\Ext^i_{\C_m}(T,X) = 0$
if and only if $X$ is in  $\add T$. The maximal $m$-rigid objects are known to have exactly $n$ non-isomorphic
indecomposable summands \cite{z}, see also \cite{w}.

Now $\C_m$ is a Krull-Schmidt category, and
we fix $$\S_m = \module H[0] \vee \module H[1] \vee \cdots \vee \module
H[m-1] \vee H[m].$$ 
Then
$\S_m$ is a fundamental domain for $\C_m$ in $\D$; 
this means that the map from isomorphism
classes of objects in $\S_m$ to isomorphism classes of objects
in $\C_m$ is bijective.  

It is easy to see the following (see \cite[Lemma 1.1]{w}).

\begin{lemma}\label{wra}
Let $T$ be an object in $\S_m$, then $\Ext_{\D}^i(T,T) = 0$ for $i= 1, \dots, m$ if and only
if $\Ext_{\C_m}^i(T,T) = 0$ for $i= 1, \dots, m$.
\end{lemma}

We then obtain the following.

\begin{proposition}\label{siltingtocluster}
Let $H$ be a finite dimensional hereditary algebra, and 
let $\S_m$ be the fundamental domain as above for the $m$-cluster category $\C_m$ of $H$.
Let $T$ be an object in $\S_m$.
Then we have the following:
\begin{itemize}
\item[(a)] $T$ is a partial silting object in $\D$ if and only if $T$ is rigid in $\C_m$.
\item[(b)] $T$ is a silting object in $\D$ if and only if $T$ is an $m$-cluster tilting object in $\C_m$.
\end{itemize}
\end{proposition}

\begin{proof}

For (a), assume first that $T$ is a partial silting object, then it is $m$-rigid by Lemma \ref{wra}.

Assume $T$ is $m$-rigid, so that $\Ext_{\D}^i(T,T) = 0$ for $i= 1, \dots, m$. Let $X,Y$ be in $\module H$,
and assume $X[r], Y[s]$ are in $\S_m$. If $r \neq m$, then $r-s <m$. So if $t>m$, then $t- (r-s) > 1$, and hence
$\Ext^t_{\D}(X[r],Y[s]) = \Hom_{D}(X,Y[t-(r-s)])= 0$.
If $r =m$, then for $t>m$ we have that $s+t-m \geq 1$, and we have $\Ext^t_{\D}(X[m],Y[s]) = 
\Hom_{D}(X,Y[(s+t-m)])$. This vanishes, since it follows from $r=m$ that $X$ must be projective.

(b): This follows from (a), since silting objects can be characterized
among partial silting objects as those with $n$ non-isomorphic 
indecomposable summands, and similarly for 
$m$-cluster tilting objects among $m$-rigid objects.  
\end{proof}

\subsection{An example}
The following example demonstrates that it is not possible to fix a 
placement of each exceptional $H$-module into $\D$ sending exceptional
sequences to silting objects.  

Consider the following simple example.
Let  $H = kQ$, where $Q$ is the quiver 
$$
\xymatrix{
1 & 2 \ar[l]  
}
$$
Then $\module H$ has three indecomposable modules, the simples $S_1$ and $S_2$, and the projective $P_2$ corresponding to
vertex 2. The complete exceptional sequences are 
$$(S_1, P_2)  \text{   ,   } (P_2, S_2) \text{   and  } (S_2,S_1)$$

Suppose the placements of $S_1, P_2, S_2$ are $S_1[a], P_2[b], S_2[c]$.
It is straightforward to check the following:
\begin{itemize}
\item[-] $S_1[a] \oplus P_2[b]$ is a silting object only if $a \leq b$
\item[-] $P_2[b] \oplus S_2[c]$ is a silting object only if $b \leq c$
\item[-] $S_2[c] \oplus S_1[a]$ is a silting object only if $c < a$
\end{itemize}
Clearly, we cannot have all three of
the above conditions simultaneously. 
Hence, it is impossible
to define a placement simultaneously for all exceptional $H$-modules taking
exceptional sequences to silting objects.

\section{Connections between mutations}

In this section we compare the mutation operations in our three settings. 

We have already discussed mutation for exceptional sequences. 
Recall that $\S_m$ denotes the full subcategory of $\D$ generated by 
$$\module H[0] \vee \module H[1] \vee \cdots \vee \module H[m-1] \vee H[m].$$
For an $H$-module $X$, we say that $X[v]$ has degree $v$, denoted by $d(X[v]) = v$.  
When we refer to the degree of an object in $\C_m$, we mean the degree
of its lifting to $\S_m$.  
We say that a map $f \colon X \to Y$ in the $m$-cluster category is a $D$-map if it is induced 
by a map $X \to Y $ in the standard domain $\S_m$.  
Otherwise we say that it
is an $F$-map.  

The following is proved in \cite{zz} (and (a) also in \cite{w}).

\begin{proposition}\label{knownfromzz}
Let $T = T_1 \oplus \cdots \oplus T_n$ be a cluster tiling object in $\C_m$.
Fix $j$ such that $1\leq j \leq n$.
\begin{itemize}
\item[(a)] There are exactly $m+1$ non-isomorphic indecomposable objects $M_0, \dots, M_m$, such that $T/T_j \oplus M_i$ is an 
$m$-cluster tilting object.
\item[(b)] The indecomposable objects $M_0,\dots,M_m$ of part (a) can 
be numbered so that there are triangles 
$$M_{i-1} \overset{f_i}{\rightarrow} B_{i} \overset{g_i}{\rightarrow} M_{i} \to $$ for $i= 1, \dots, m$, 
with indices computed modulo $m$,
such that each $f_i$ is a minimal left $\add T/T_j$-approximation in $\C_m$ and each $g_i$ is a 
minimal right $\add T/T_j$-approximation in
$\C_m$. 
\item[(c)] For $0 \leq i \leq m-1$, we have that $d(M_i) \leq d(M_{i+1})$. 
\item[(d)] For $0 \leq i \leq m$, we have that $i-1 \leq d(M_i)\leq i$.
\end{itemize} 
\end{proposition}

By Proposition \ref{siltingtocluster} it follows that an almost complete $m$-cluster tilting object 
$T/T_j$ can be considered 
an almost complete silting object in $\D$, lying in $\S_m$. It also follows that the complements 
of $T/T_j$ in $\C_m$ coincide with those complements of $T/T_j$ (as an almost complete silting object)
which lie in $\S_m$.

In this section we show that mutation of $m$-cluster tilting objects can be naturally interpreted as a mutation of 
silting objects. We also point out some simple examples demonstrating that 
connections to mutation of exceptional sequences are more involved. 

\subsection{Mutation of silting objects and cluster tilting objects}

We begin with a lemma, which is an adaption of a similar statement for
tilting modules or cluster tilting objects. The proof is also similar.

\begin{lemma}\label{silting-exchange}
Let $\overline{T}$ be an almost complete silting object, and let 
$M$ be a complement.
\begin{itemize}
\item[(a)]
Let  
\begin{equation}\label{exex}
M^{\ast} \overset{f}{\rightarrow} B \overset{g}{\rightarrow} M \to 
\end{equation}
be a triangle such that
$g$ is a minimal right $\add \overline{T}$-approximation. Then $M^{\ast}$ is a complement to $\overline{T}$, with
$M^{\ast} \not \simeq M$, and with $f$ a minimal left $\add \overline{T}$-approximation.
\item[(b)]
Let  
\begin{equation}\label{exex2}
M \overset{f'}{\rightarrow} B' \overset{g'}{\rightarrow} M^{\sharp} \to 
\end{equation}
be a triangle such that
$f'$ is a minimal left $\add \overline{T}$-approximation. Then $M^{\sharp}$ is a complement to $\overline{T}$, with
$M^{\sharp} \not \simeq M$, and with $g'$ a minimal right $\add \overline{T}$-approximation.
\end{itemize}
\end{lemma}

\begin{proof}
We prove only (a); the proof of (b) is dual.
By applying $\Hom_{\D}(\overline{T}, \ )$ and
$\Hom_{\D}(\ , \overline{T})$ to the triangle (\ref{exex}), and considering the corresponding long exact sequences,
we get that $\Ext^i(\overline{T},M^{\ast}) = \Ext^i(M^{\ast}, \overline{T}) = 0$ for $i>0$. Note that, in particular,
we use that $g$ is a right $\add \overline{T}$-approximation to obtain that $\Ext^1(\overline{T},M^{\ast})$ 
vanishes.

From the exact sequence
$$\Hom(B, \overline{T}) \to \Hom(M^{\ast}, \overline{T}) \to \Ext^1(M, \overline{T})= 0$$
it follows that $f$ is a left $\add \overline{T}$-approximation.

It is clear that $f$ must be left minimal, since $M$ is indecomposable, and not a direct summand in $B$.

We claim that $M^{\ast}$ must be indecomposable. Assume to the contrary that $M^{\ast} = U \oplus V$,
with $U$ and $V$ both non-zero. Consider minimal left $\add \overline{T}$-approximations 
$f_1 \colon U \to B_1$ and $f_2 \colon V \to B_2$, and complete to get triangles
$$U \to B_1 \to X \to \text{  and  } V \to B_2 \to Y \to $$ 
Since the direct sum of the triangles is the triangle (\ref{exex}),
it follows that $M \simeq X \oplus Y$, and hence either $X= 0$ or $Y= 0$. If $X=0$, then $B_1 \to 0$ is
a direct summand of $f \colon B \to M$, which contradicts that $g$ is right minimal.
Similarly $Y=0$ leads to a contradiction, so we get that $M^{\ast}$ is indecomposable.

Clearly $M^{\ast}$ is not in $\add \overline{T}$, since then $f$ would be an isomorphism, contradicting
$M \neq 0$.

We are now left with showing $\Ext^i(M^{\ast}, M^{\ast}) = 0$ for $i>0$. 
Applying $\Hom(\ , M)$ to the triangle (\ref{exex}), considering the corresponding long exact sequence,
and using that
$\Ext^1(M,M) = 0$, we get an epimorphism
\begin{equation}\label{firstepi} \Hom(B,M) \to \Hom(M^{\ast}, M) \end{equation}
So any map $h \colon M^{\ast} \to M$ factors through $f \colon M^{\ast} \to B$.

Then we apply $\Hom(M^{\ast}, \ )$ to the triangle (\ref{exex}) and consider
the resulting exact sequence
$$\Hom(M^{\ast}, B) \to \Hom(M^{\ast},M) \to \Ext^1(M^{\ast},M^{\ast}) \to  \Ext^1(M^{\ast},B)$$
Since the rightmost term vanishes, it is sufficient to show that $\Hom(M^{\ast}, B) \to \Hom(M^{\ast},M)$
is surjective in order to obtain $\Ext^1(M^{\ast}, M^{\ast}) = 0$.
To see that any map $h \colon M^{\ast} \to M$ factors through $g \colon B \to M$, consider
the commutative diagram
$$
\xymatrix{
& & B \ar^t[d] \ar@{-->}_s[dl] \\
M^{\ast} \ar_h[d] \ar^f[r]& B \ar_g[r] \ar@{-->}^t[dl]& M \\
M & & 
}
$$

Here $t$ is obtained from the epimorpism (\ref{firstepi}), and we get $s \colon B \to B$ by using
that $g$ is a right $\add \overline{T}$-approximation. So $h= tf = gsf$.

Applying $\Hom(\ , M)$ to the triangle (\ref{exex}), and considering the resulting long exact
sequence, we obtain $\Ext^i(M^{\ast},M)=0$ for $i \geq 1$.

It then follows from the long exact sequence obtained by 
applying $\Hom(M^{\ast}, \ )$ to the triangle (\ref{exex}) that
$\Ext^i(M^{\ast}, M^{\ast}) = 0$ for $i>1$. 
This finishes the proof.
\end{proof}

The triangles appearing 
in Lemma \ref{silting-exchange} are called {\em exchange triangles} in $\D$.
We need the following observation.

\begin{lemma}\label{forgetting}
If a minimal left $\add \overline{T}$-approximation $\alpha$ in $\C_m$ is a $D$-map, then it is also 
a minimal left $\add \overline{T}$-approximation in $\D$.  
\end{lemma}

For our main result in this section, we also need the following two lemmas.

\begin{lemma}\label{dmaps}
Let $\overline{T}$ be an almost complete silting object 
with a complement $M$.
Let $m \geq 2$, and
let 
\begin{equation}\label{exinc}
M^{\ast} \overset{f}{\rightarrow} B \overset{g}{\rightarrow} M \to 
\end{equation} 
be an exchange triangle in $\C_m$ with $B \neq 0$.
Then $d(M^{\ast}) \leq d(M)$ if and only if both $f$ and $g$ are $D$-maps.

\end{lemma}

\begin{proof}
If both $f$ and $g$ are $D$-maps, then clearly $d(M^{\ast}) \leq d(M)$.

Note that if $X \to Y$ is not a $D$-map, with $X$ and $Y$ indecomposable, we must have $d(Y) = 0$ and $d(X) \geq m-1$.

Assume $f$ is not a $D$-map. 
Then the middle term 
$B$ must have a summand of degree $0$, and hence $d(M) \leq 1$.
Also, we have that $d(M^{\ast})$ must be $m-1$ or $m$. 
If $m \geq 3$, then clearly $d(M^{\ast}) > d(M)$.  
Consider the case $m=2$. If $d(M^{\ast})= 2$, then trivially  $d(M^{\ast}) > d(M)$. Assume $d(M^{\ast})= 1$.
We already know that $d(M) \neq 2$, and now we claim that $d(M) \neq 1$. Assume $d(M)=1$. The exchange triangle (\ref{exinc}) must be
induced either (i) by an exact sequence in $\module H$, or (ii) by a triangle 
$$M^{\ast} \overset{f}{\rightarrow} B' \overset{g}{\rightarrow} \tau^{-1} M[2] \to $$ in the derived category $\D$.
Then (i) is not possible, by the assumption that $f$ is an $F$-map. Also 
(ii) is not possible, since a map $\tau^{-1} M[2] \to M^{\ast}[1]$
is zero in the derived category, and hence the triangle must split, which is a contradiction. 
Hence $d(M) = 0$.

Assume $f$ is a $D$-map, but $g$ is not. 
Since $g$ is not a $D$-map, then $B$ must have a direct summand in degree $m-1$ or $m$,
and $d(M) = 0$. Since $f$ is $D$-map we have $d(M^{\ast}) \geq m-2$.
So for $m \geq 3$, we clearly have $d(M^{\ast}) > d(M)$. 
Consider the case $m = 2$, and assume $d(M^{\ast}) = d(M) = 0$. 
Then $B$ must have a summand of degree $>0$, and as above this implies that the 
exchange triangle (\ref{exinc}) must be
a triangle 
$$M^{\ast} \overset{f}{\rightarrow} B' \overset{g}{\rightarrow} \tau^{-1} M[2] \to $$ in the derived category $\D$,
which is not possible, since a map $\tau^{-1} M[2] \to M^{\ast}[1]$
is zero in the derived category, and hence the triangle must split, which is a contradiction. 
\end{proof}

We can now prove the main result in this section.

\begin{theorem}
Let $T = T_1 \oplus \cdots \oplus T_n$ be a basic silting object in $\D = D^b(H)$, where the $T_i$ are indecomposable 
and $n$ is the number of isomorphism-classes of simple $H$-modules. We assume without loss of generality that 
$T$ is in $\D^+$. Let $m$ be an integer such that $T$ is in the fundamental domain $\S_m$ of the $m$-cluster category $\C_m$.
Fix an indecomposable direct summand $T_i$, with $1 \leq i \leq n$.
\begin{itemize}

\item[(a)] The almost complete silting object $T/T_i$ has exactly $m+1$ non-isomorphic complements 
$M_0, M_1, \dots, M_m$ lying in $\S_m$, 
\sloppy which are ordered such that $d(M_i) \leq d(M_{i+1})$. 

\item[(b)] For each $j= 1, \dots, m$, there is a triangle in $\D$
\begin{equation}\label{exchange} 
M_{j-1} \overset{f_j}{\rightarrow} B_j \overset{g_j}{\rightarrow} M_j \to,
\end{equation}
such that $f_j$ is a minimal left $\add T/T_i$-approximation, 
and $g_j$ is a minimal right $\add T/T_i$-approximation. 

\item[(c)] The triangles $(\ref{exchange})$ are also exchange triangles in $\C_m$.

\end{itemize}
\end{theorem}

Note that in the $m$-cluster category $\C_m$, there is an additional exchange triangle  
$$M_m \to B_0 \to M_0 \to$$
which is not a triangle in $\D$.

We have the following direct consequence, with assumptions as in the theorem.

\begin{corollary}
$T/T_i$ has a countably infinite number of non-isomorphic complements $M_i$ for $i \in \mathbb{Z}$.
In particular, there are complements $M_{-1} $ and $M_{m+1}$, such that 
$M_j  \simeq M_{-1}[j+1]$ for $j < -1$, and $M_j = M_{m+1}[j-(m+1)]$ for $j > m+1$. 
\end{corollary}

\begin{proof}[Proof of theorem]
Note that by Proposition \ref{siltingtocluster}, 
the complements of $T/T_i$ as an almost complete silting object and as an 
almost complete $m$-cluster tilting object coincide. 
From Proposition \ref{knownfromzz}, 
there are, up to isomorphism, exactly $m+1$ complements 
$M_0, M_1, \dots, M_m$ of $T/T_i$ as an almost complete $m$-cluster tilting object, 
and they can be ordered 
such that for $j= 0, \dots, m-1$, there are exchange triangles $M_j \to B_{j+1} \to M_{j+1} \to$ 
in $\C_m$ and $d(M_{j}) \leq d(M_{j+1})$.

Assume first $m \geq 2$.
All the maps 
$M_{j}  \to B_{j+1}$ and $B_{j+1} \to M_{j+1}$ 
are then $D$-maps by Lemma \ref{dmaps}. Then $M_0 \to B_1$ is by Lemma \ref{forgetting} a minimal 
left $\add \overline{T}$ approximation in $\D$. Hence, by Lemma \ref{silting-exchange} (b) it follows that 
the exchange triangle $M_0 \to B_1 \to M_1$ in $\D$ coincides with the exchange triangle in $\C_m$. 

Iterating this for $M_i \to B_{i+1}$,  
all of (a), (b), and (c) follow.

Consider now the case $m=1$, and the exchange sequence
\begin{equation}\label{mut}
M_0 \to B_{1} \to M_1 \to 
\end{equation}
in $\C_m$.

\noindent (i) If $d(M_0) = d(M_1) = 0$ (up to reordering $M_0$ and $M_1$) 
the sequence (\ref{mut}) is induced by an exact sequence in $\module H$.

\noindent (ii) If $d(M_0) = 0$ and $d(M_1) =1$, it is straightforward to check that the sequence (\ref{mut}) must be induced by
a sequence $M_0 \to B_0 \to P[1] \to $ in the derived category, with both 
$M_0 \to B_{1}$ and $B_{1} \to P[1]$ being $D$-maps, and with $P$ an indecomposable projective.

Hence in both case (i) and (ii) both (a), (b) and (c)
follow from combining Lemmas \ref{silting-exchange} and
\ref{forgetting} as in the above case.
\end{proof}

\subsection{Example}
Let $\E = \E_0$ be a complete exceptional sequence and let $\E_i = \mu_i(\E)$ for $i= 1, \dots , n-1$.
A natural question is: Is it possible to embed these $n$ exceptional sequences simultaneously in $\S_m$
for some value of $m$, such that the following are satisfied 
\begin{itemize}
\item each $X$ occurring in some $\E_i$ is mapped to a fixed $X[j_X]$ 
\item each $\E_i$ is mapped to an $m$-cluster tilting object $U_i$, in such a way that
\item the $m$-cluster tilting object $U_i$ is obtained by mutation of $U_0$ in $\C_m$ 
\end{itemize}

The following example demonstrates that such an embedding is not possible 
in generalß.

Let $Q$ be the quiver
$$
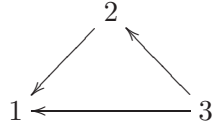
\xymatrix{
& 2 \ar[dl] & \\
1& & 3 \ar[ll] \ar[ul]
}
$$
and $H = kQ$. Let $P_i$ be the indecomposable projective corresponding to vertex $i$.
Then $\E = (P_1,P_2, P_3)$ is a complete exceptional sequence.
We have $\mu_1(\E)  = (P_2, S_2, P_3)$ and $\mu_2(\E) = (P_1, P_3, R )$, where
$S_1$ is the simple corresponding to vertex 1, and $R$ is a regular indecomposable module with composition factors
$S_3$ and $S_1$.

Assume the images of the $P_i$ are $P_1[a], P_2[b], P_3[c]$.
We must have that the image $S_2[d]$ of $S_2$ is obtained by mutation of the cluster tilting object 
$P_1[a] \oplus P_2[b] \oplus P_3[c]$ at $P_1[a]$. 
Hence we must have $\Hom(P_1[a], P_2[b]) \neq 0$  and $\Hom(P_1[a], P_3[c])= 0$, so we must have $a=b$ and $a \neq c$.  
Considering mutation at $P_2[b]$ we get similarly that $\Hom(P_2[b], P_2[c]) \neq 0$, so we must have $b=c$, 
and we have a contradiction. 

\section{Placement of an almost complete
exceptional sequence and its complements}\label{sec4}

As before, let $H$ be a hereditary algebra with $n$ isomorphism classes of simple modules.
In this section we show that for a fixed almost complete exceptional sequence $\E$ in $\module H$, and its set of $n$ complements, 
there is a natural interpretation of $\E$ as an almost complete silting object $\widehat Aß$ in $\S_{n-1}$, in such a way that  
the complements and exchange sequences correspond.

Consider an almost complete exceptional sequence
$\E = (A_0, A_1, \dots, A_{n-2})$.
Let $\C = \{ C_0, C_1, \dots, C_{n-1} \}$ be the complements of $\E$,
such that each
$$\E_i = (A_0, A_1, \dots, C_i, A_i, \dots, A_{n-2})$$
is a complete exceptional sequence.
Recall from section \ref{basics} that for each 
$i = 0, \dots, n-2$, there is a mutation $\mu_i$, replacing  $(C_i, A_i)$ with $(A_i, C_{i+1})$ and that
there is an induced triangle in $\D$ 
\begin{equation}\label{m}
C_{i} \overset{f_i}{\rightarrow} A_i^{r_i}[v] \overset{g_i}{\rightarrow} C_{i+1}[w] \to ,
\end{equation}
with $f_i$ a minimal left $(\operatorname{thick} A_i)$-approximation, with $v,w \in\{0,1\}$. 
Note that $f_i$ might be the zero-map (i.e. $r_i =0$), in which case $C_{i+1} = C_i$ and $w=1$.

We want to define objects $\widehat{C_i}, \widehat{A_i}$ inside the standard domain $\S_{n-1}$, such that the 
following hold.

\begin{itemize}
\item[-] $\widehat{C_i}= C_i[t_i] $ and $\widehat{A_i} = A_i[u_i]$, for integers $t_i, u_i$ 
such that $\widehat{C_i}, \widehat{A_i}$ in $\S_{n-1}$.

\item[-] $\bigoplus_{i=0}^{n-2} \widehat{A_i}$ is an almost complete silting object with complements $\widehat{C_i}$.

\item[-] There are exchange triangles $\widehat{C_i} \to 
\widehat{A_i^{r_i}} \to \widehat{C_{i+1}} \to$ given by shifting the 
mutation triangles (\ref{m}).
\end{itemize}

\begin{theorem}\label{main}
\sloppy Given an almost complete exceptional sequence
$\E = (A_0, A_1, \dots, A_{n-2})$, 
there exists a placement of the objects $A_i$, $C_i$ in $\S_m$ with the
above properties.  
\end{theorem}

\begin{proof}
The proof is by induction on $n$. In case $n=1$, the almost complete exceptional sequence $\E$ is empty, and has a unique complement 
$C_0 =  H$. The standard domain $\S_0$ also consists of only the object $H$, and a placement trivially exists.

Assume now $n >1$, and
consider the category $$A_{n-2}^{\perp} = \{U \in \module H \mid \Hom(A_{n-2},U) = \Ext^1(A_{n-2},U) = 0 \}.$$ 
By \cite{cb,r} there is an hereditary algebra $H'$ such that $\module H' \simeq A_{n-2}^{\perp} $, and
$H'$ has $n-1$ isomorphism classes of simple modules.

Since each $A_i$ for $0 \leq i \leq n-3$ and each $C_i$ for $0 \leq i \leq n-2$ is in $A_{n-2}^{\perp} $, we have
by induction a placement for these in the fundamental domain for $H'$, which as a full additive subcategory of $\D$ is generated by
$A_{n-2}^{\perp} [j]$ for $0 \leq j \leq n-3$, and $X[n-2]$, where $X$ is the direct sum of all indecomposable $\Ext$-projectives
of $A_{n-2}^{\perp}$.   
 
We now define $\widehat{A_{n-2}}$ and $\widehat{C_{n-1}}$ according to the following rules.

\begin{itemize}
\item[(P1)] If $\Hom(C_{n-2}, A_{n-2}) \neq 0$, then let $d(\widehat{A_{n-2}}) = d(\widehat{C_{n-2}})$. 
\item[(P2)] If $\Ext^1(C_{n-2}, A_{n-2} ) \neq 0$, then let $d(\widehat{A_{n-2}}) = d(\widehat{C_{n-2}})+1$.
\item[(P3)] Otherwise let $d(\widehat{A_{n-2}})  \geq n-3$ and minimal such that the following is satisfied
\begin{itemize}
\item[(i)]  $d(\widehat{A_{n-2}})  \geq d(\widehat{C_{n-2}})$,
\item[(ii)]  $d(\widehat{A_{n-2}}) > d(\widehat{A_{n-3}})$ if $\Ext^1(A_{n-3}, A_{n-2} ) \neq 0$, 
\item[(iii)] $d(\widehat{A_{n-2}}) > d(\widehat{A_{n-4}})$ if $\Ext^1(A_{n-4}, A_{n-2} ) \neq 0$.  
\end{itemize}
\end{itemize}
In case (P1) or (P2), $\widehat{C_{n-1}}$ is defined such that  
we have a triangle $$\widehat{C_{n-2}} \to
\widehat{A_{n-2}^{r_{n-2}}} \to \widehat{C_{n-1}} \to.$$
In case (P3), we let $\widehat{C_{n-1}} = \widehat{C_{n-2}}[1]$.  

We claim that with this 
definition of $\widehat{A_{n-2}}$ and $\widehat{C_{n-1}}$ we have that the following hold.

\begin{itemize}
\item[$\mathcal{U}_n$:] $\widehat{A_{n-2}}$ is in $\S_{n-1}$.
\item[$\mathcal V_n$:] $\widehat{A} = \oplus_{i= 0}^{n-2} \widehat{A_{i}}$ is an almost complete silting object.
\item[$\mathcal W_n$:] $\widehat{C_{i}}$ for $i = 0,  \dots , n-1$ are the complements of $\widehat{A}$ inside $\S_{n-1}$.
\item[$\mathcal X_n$:] If $d(\widehat{A_{n-2}})   = d(\widehat{C_{n-1}})$, then there is a directed path from $A_{n-2}$ to $C_{n-1}$ 
in the Hom-Ext quiver $\Q$ associated to the exceptional sequence 
$$(A_0, \dots , A_{n-2}, C_{n-1})$$ with the property that 
the last arrow on this path corresponds to a morphism to $C_{n-1}$.
\item[$\mathcal Y_n$:] $d(\widehat{C_{n-1}}) \geq d(\widehat{A_{n-2}})  \geq d(\widehat{C_{n-2}})$
\end{itemize}

It is clear that all statements hold for $n=1$. It is also clear that
the claim of the theorem follows from $\mathcal{U}_n$, $\mathcal V_n$, 
$\mathcal W_n$.
We need also to include $\mathcal X_n$ and  $\mathcal Y_n$ as part of our induction argument.

We first need a lemma. In \cite{bong} it was shown that for any exceptional non-projective
module $Y$, there is a (Bongartz-)complement $W$, with the properties that
\begin{itemize}
\item[-] $Y \oplus W$ is tilting,
\item[-] $\Ext^1(Y,U) = 0$ implies $\Ext^1(W,U) = 0$ for any $U$ in $\module H$ and  
\item[-] $\Hom(Y,W) =0$.
\end{itemize}

\begin{lemma}\label{bonga}
Assume $A_{n-2}$  is not projective, then the direct sum $X$ of one copy of each
of the indecomposable Ext-projectives in
$A_{n-2}^{\perp}$, up to isomorphism, equals the Bongartz-complement of $A_{n-2}$.
\end{lemma}

\begin{proof}
Let $W$ be the Bongartz complement of $A_{n-2}$ in $\module H$. Then it follows directly
from the properties of the Bongartz complement that $W$ is in  $A_{n-2}^{\perp}$, and is 
Ext-projective in this subcategory. Since it is known that  $A_{n-2}^{\perp}$ is equivalent to a module category
with $n-1$ isomorphism classes of simples, the statement follows.
 \end{proof}

By induction we can assume $\mathcal U_m$--$\mathcal Y_m$ hold for $m<n$. We need to prove that they hold also for $n$.  \\
\\
\noindent Proof of  $\mathcal{U}_n$: \\ 
It is clear that we only need to consider the case where  $A_{n-2}$ is placed according to P2 or P3.
Assume first  $\Ext^1(C_{n-2}, A_{n-2} ) \neq 0$, so that  $d(\widehat{A_{n-2}}) = d(\widehat{C_{n-2}})+1$ according to P2.
If $d(\widehat{C_{n-2}}) \leq n-3$, then $d(\widehat{A_{n-2}}) \leq n-2$. 
So assume  
$d(\widehat{C_{n-2}}) = n-2$ and $d(\widehat{A_{n-2}}) = n-1$. 
If $A_{n-2}$ is non-projective, then $C_{n-2}$ must be a summand in $X$, and hence Ext-projective in $A_{n-2}^{\perp}$. Hence,
by Lemma \ref{bonga}  we have that $\Ext^1(C_{n-2},A_{n-2}) = 0$, contradicting our assumption.
If $A_{n-2}$ is projective, then 
 $\widehat{A_{n-2}} $ is in $\S_n$, and we are done.

Assume $\widehat{A_{n-2}}$  is placed according to case P3. We need to show that if $d(\widehat{A_{n-2}}) = n-1$,
then $A_{n-2}$ is projective. 
We have that $d(\widehat{C_{n-2}}) \leq n-2$, by $\mathcal W_{n-1}$. Therefore, if
$d(\widehat{A_{n-2}}) = n-1$, this must be forced by (ii) or (iii) in P3. 
However $d(\widehat{A_{n-4}}) \leq n-3$ by $\mathcal U_{n-2}$, so we need only consider option (ii), i.e. $\Ext^1(A_{n-3}, A_{n-2} ) \neq 0$.
We must have $d(\widehat{A_{n-3}}) = n-2$, hence $A_{n-3}$ is a summand in $X$, the sum of the Ext-projectives in $A_{n-2}^{\perp}$.
If $A_{n-2}$ was non-projective, then $\Ext^1(A_{n-3}, A_{n-2}) = 0$ by Lemma \ref{bonga}, and we have a contradiction.
Hence, we have that  $A_{n-2}$ is projective, and therefore
 $\widehat{A_{n-2}} $ is in $\S_{n-1}$, and we are done. \\
 \\
\noindent Proof of $\mathcal V_n$: \\
By induction $\widehat{A'} = \bigoplus_{i= 0}^{n-3}  \widehat{A_{i}}$ is a partial silting object in $\D(H')$, lying inside the
fundamental domain for $H'$. Clearly, $\widehat{A'}$ is also a partial silting object in $\D(H)$.

\sloppy We need to show that if $d(\widehat{A_{n-2}}) =d(\widehat{A_{n-j}}) $ for some $j \geq 3$,
then $\Ext^1(A_{n-j}, A_{n-2}) = 0$. Note that for $j \geq 5$, we have that $d(\widehat{A_{n-j}}) \leq n-4$ by $\mathcal U_{n-3}$, while
$d(\widehat{A_{n-2}})  \geq n-3$, so we only need to consider the cases $j=3,4$.

We first consider the case $j=3$.  Assume $d(\widehat{A_{n-2}}) =d(\widehat{A_{n-3}}) $. By $\mathcal Y_{n-1}$, we have 
$d(\widehat{C_{n-2}}) \geq d(\widehat{A_{n-3}})$, and by considering the different cases of our placement algorithm 
$d(\widehat{A_{n-2}}) \geq d(\widehat{C_{n-2}})$. Hence we must have 
$d(\widehat{A_{n-2}}) = d(\widehat{C_{n-2}}) = d(\widehat{A_{n-3}})$. 

Using $\mathcal X_{n-1}$, we find, in the Hom-Ext quiver $\Q$ associated to the exceptional sequence $(A_0, \dots, A_{n-3}, C_{n-2})$ in $\module H'$,
a directed path from $A_{n-3}$ to $C_{n-2}$. 
Assume first we are in case P1, so that we have a non-zero map
from  $C_{n-2}$ to $A_{n-2}$. If $\Ext^1(A_{n-3}, A_{n-2}) \neq 0$, this would give us an arrow from $A_{n-2}$ to $A_{n-3}$ 
in the Hom-Ext quiver $\Q$, and hence an oriented cycle. By Theorem \ref{acyclic}, such cycles do not exist, so we must have 
$\Ext^1(A_{n-3}, A_{n-2}) = 0$, which is what we wanted to show.  
Now assume we are in case P3. Then, since $d(\widehat{A_{n-2}}) = d(\widehat{A_{n-3}})$, we must 
have that $\Ext^1(A_{n-3}, A_{n-2}) = 0$, according to (ii). 
It is also clear we cannot be in case P2, since $d(\widehat{A_{n-2}}) = d(\widehat{C_{n-2}})$.

We now consider the case $j=4$. Assume $d(\widehat{A_{n-2}}) =d(\widehat{A_{n-4}})$.
Using $\mathcal Y_{n-1}$ and 
$\mathcal Y_{n-2}$, this implies that  
$$d(\widehat{A_{n-2}}) = d(\widehat{C_{n-2}}) = d(\widehat{A_{n-3}}) = d(\widehat{C_{n-3}}) = d(\widehat{A_{n-4}}).$$

We can apply $\mathcal X_{n-2}$ to find a directed path  from $A_{n-4}$ to $C_{n-3}$
in the Hom-Ext quiver $\Q$ of $$(A_0,\dots,A_{n-4},C_{n-3}).$$
Since $d(\widehat{C_{n-2}}) =  d(\widehat{C_{n-3}})$ , there must be an exchange triangle
$$\widehat{C_{n-3}}  \to \widehat{A_{n-3}^{r_{n-3}}}  \to \widehat{C_{n-2}} \to $$ 
which is induced by an exact sequence in $\module H$. 
Since $C_{n-3}$ embeds into $A_{n-3}^{r_{n-3}}$, we can compose the 
final morphism in the path from $A_{n-4}$ to $C_{n-3}$ with this inclusion, 
obtaining a path from $A_{n-4}$ to $A_{n-3}$, onto which the arrow from
$A_{n-3}$ to $C_{n-2}$ can be added, giving us a path from $A_{n-4}$ to
$C_{n-2}$.

Assume we are in case P1.
We then have
a map from $C_{n-2}$ to $A_{n-2}$. Hence, if $\Ext^1(A_{n-4}, A_{n-2}) \neq 0$, 
this would give us an arrow from $A_{n-2}$ to $A_{n-4}$, and hence an oriented cycle
in the Hom-Ext quiver $\Q$. Therefore we must have $\Ext^1(A_{n-4}, A_{n-2}) = 0$.
Assume we are in case P3. Then, since 
$d(\widehat{A_{n-2}}) = d(\widehat{A_{n-4}})$, we have that $\Ext^1(A_{n-4}, A_{n-2}) = 0$, according to (iii). 
It is also clear we cannot be in case P2, since $d(\widehat{A_{n-2}}) = d(\widehat{C_{n-2}})$.
Hence we are done with $\mathcal V_n$. \\
\\
\noindent Proof of $\mathcal W_n$: \\
We first claim that $\widehat{C_{n-2}}$ is a complement to $\widehat{A}$. By induction, we only need to show that it is compatible
with $\widehat{A_{n-2}}$. But this follows directly from P2 in our placement algorithm.

Next, we claim that all $\widehat{C_{i}}$ for $i= 0, \dots, n-3$ are also complements.
For this, consider the triangles 
$$\widehat{C_{i-1}} \to \widehat{{A_{i}}^{r_i}} \to \widehat{C_{i}} \to$$
for $i=1, \dots, n-2$.
The maps $\widehat{{A_{i}}^{r_i}} \to \widehat{C_{i}}$ are by assumption minimal right $\widehat{A'}$-approximations in the fundamental domain for
$A_{n-2}^{\perp}$, which is a full subcategory of $\D$. Hence they are also minimal right $\widehat{A'}$-approximations in $\D$. We have that
$\Hom(A_{n-2}, C_{j}) = \Ext^1(A_{n-2}, C_{j}) = 0$ for $j \leq n-2$ by our assumptions. Hence 
the maps $\widehat{A_{i}} \to \widehat{C_{i}}$ 
are also right $\widehat{A}$-approximations. Using Lemma \ref{silting-exchange}, we obtain that all $\widehat{C_{i}}$ for $i \leq n-2$ are complements.

Finally, we claim that $\widehat{C_{n-1}}$ is also a complement.
This follows directly from our placement algorithm, since 
the left $\widehat{A_{n-2}}$-approximation of $\widehat{C_{n-2}}$ is indeed a left $\widehat{A}$-approximation.
This follows since $\widehat{C_{n-2}}$ does not have any non-zero maps to $\widehat{A'}$.\\
\\
\noindent Proof of $\mathcal X_n$: \\
Assume first $\Hom(C_{n-2}, A_{n-2}[j]) \neq 0$ for $j=0$ or $j=1$.
Then $\widehat{A_{n-2}}, \widehat{C_{n-1}}$ are placed using P1 or P2,
so if $d(\widehat{A_{n-2}}) = d(\widehat{C_{n-1}})$, then clearly there
is a map $A_{n-2} \to C_{n-1}$. 

So we can assume $\Hom(C_{n-2}, A_{n-2}[j]) = 0$ for $j=0,1$, 
and hence $C_{n-1} = C_{n-2}$.
Using $\mathcal W_{n-1}$ and Proposition \ref{knownfromzz}(d), we have that $d(\widehat{C_{n-2}}) \geq n-3$,
and hence $d(\widehat{C_{n-1}}) \geq n-2$. 
Hence 
$d(\widehat{A_{n-2}}) \geq n-2$ by assumption. Since 
$d(\widehat{A_{n-2}}) \neq n-3$, one of the following must 
hold
\begin{itemize}
\item[(i')] $d(\widehat{C_{n-2}}) =d(\widehat {A_{n-2}})= n-2$
\item[(ii')] $\Ext^1(A_{n-3},A_{n-2}) \neq 0$ and (ii) in P3 applies
\item[(iii')] $\Ext^1(A_{n-4},A_{n-2}) \neq 0$ and (iii) in P3 applies
\end{itemize}

Assume first (i'). In this case we have
$d(\widehat{A_{n-2}}) = d(\widehat{C_{n-2}}) = n-2$ and
$d(\widehat{C_{n-2}}) < d(\widehat{C_{n-1}})$, hence we cannot have
$d(\widehat{A_{n-2}}) =   d(\widehat{C_{n-1}})$, so there is nothing to show.

Assume now (ii').
By assumption 
$\widehat{C_{n-1}} = \widehat{C_{n-2}}[1]$, so we have that
$d(\widehat{C_{n-2}}) = d(\widehat{A_{n-3}})$. Hence, we can apply
$\mathcal X_{n-1}$, to get a directed path from $A_{n-3}$ to $C_{n-2}$
in the Hom-Ext quiver of the exceptional sequence 
$$(A_0, \dots, A_{n-3}, C_{n-2})$$
and such that this path ends with an arrow corresponding to a homomorphism.
Since $C_{n-2} = C_{n-1}$, this gives us a 
a directed path from $A_{n-3}$ to $C_{n-1}$
in the Hom-Ext quiver of the exceptional sequence 
$$(A_0, \dots, A_{n-3}, A_{n-2}, C_{n-1}).$$
By assumption,  $\Ext^1(A_{n-3},A_{n-2}) \neq 0$, so we can add 
$A_{n-2}\rightarrow A_{n-3}$ to the front of the directed path, producing
the directed path we need.
 
Assume now (iii').
By $\U_{n-2}$, 
we have $d(\widehat{A_{n-4}}) \leq n-3$.
Then, since (iii) forced $d(\widehat{A_{n-2}}) > n-3$, we must have
$d(\widehat{A_{n-4}}) = n-3$, and we have $d(\widehat{A_{n-2}})= n-2$.
By assumption, then $d(\widehat{C_{n-1}})= n-2$, and
since $C_{n-2} = C_{n-1}$, we have  $d(\widehat{C_{n-2}})= n-3$.

By $\mathcal Y_{n-1}$, we then have that also $d(\widehat{C_{n-3}})= n-3$.
Since  $d(\widehat{C_{n-2}})= d(\widehat{C_{n-3}})$,
we must have a short exact sequence
$$0 \to C_{n-3} \to A_{n-3}^{r_{n-3}}  \to C_{n-2} \to 0$$

By $\mathcal X_{n-2}$
we have a path $\rho$ from $A_{n-4}$ to $C_{n-3}$
in the Hom-Ext quiver of the exceptional sequence
$$(A_0, \dots, A_{n-4}, C_{n-3})$$ ending in an arrow, say
$A_j\rightarrow C_{n-3}$ for some $j$, and 
where this arrow is either an m-arrow or an
e-arrow (i.e., it represents a morphism from $A_j$).  
We can prepend to $\rho$ the x-arrow from $A_{n-2}$ to $A_{n-4}$,
resulting in a directed path from $A_{n-2}$ to $C_{n-3}$.  

We know that $C_{n-3}$ injects into $A_{n-3}^{r_{n-3}}$, so we can
compose the map from $A_j$ to $C_{n-3}$ with this inclusion, to obtain
an arrow in the Hom-Ext quiver from $A_j$ to $A_{n-3}$.  There is also an
arrow from $A_{n-3}$ to $C_{n-2}$.  Thus we can remove the last step of
$\rho$ and replace it by $A_j\rightarrow A_{n-3}\rightarrow C_{n-2}=
C_{n-1}$, obtaining a directed path as desired.  
\\
\\
\noindent Proof of $\mathcal Y_n$: \\
The inequality 
$d(\widehat{A_{n-2}}) \geq d(\widehat{C_{n-2}})$ follows directly from (P1), (P2)
and part(i) of (P3).

The inequality 
$d(\widehat{C_{n-1}}) \geq d(\widehat{A_{n-2}})$
is clearly satisfied if we are in case (P1) or (P2).
So assume we are in case (P3), and hence $C_{n-2} = C_{n-1}$,
and $d(\widehat{C_{n-2}}) = d(\widehat{C_{n-1}}) -1$.

Proposition \ref{knownfromzz} and $\mathcal W_{n-1}$ together imply that
$d(\widehat{C_{n-2}})\geq n-3$. 
Hence, we have 
$d(\widehat{A_{n-2}})  \geq n-3$.
Assume that the desired inequality is not satisfied, so we have
$d(\widehat{C_{n-1}}) < d(\widehat{A_{n-2}})$. 
Then we must have $d(\widehat{A_{n-3}}) = d(\widehat{C_{n-1}}) > 
 d(\widehat{C_{n-2}})$, which contradicts $\mathcal Y_{n-1}$.
This finishes the proof for $\mathcal Y_n$.

\end{proof}

\end{document}